\documentclass{jloganal}
\usepackage{pinlabel}
\usepackage{amsmath,amssymb,mathtools}
\usepackage{graphicx}
\usepackage{listings}
\usepackage{xcolor}
\usepackage{pifont}   % for \ding
\usepackage{enumitem}
\usepackage{booktabs}
\usepackage{array}
\usepackage{subcaption}
\usepackage{tikz}
\usepackage{proof}
\usepackage{float}
\usetikzlibrary{decorations.markings}

\tikzset{
    strike through/.style={
        postaction=decorate,
        decoration={
            markings,
            mark=at position 0.5 with {
                \draw[-] (0.2em,0.2em) -- (-0.2em,-0.2em);
            }
        }
    }
}

\definecolor{codegreen}{rgb}{0,0.6,0}
\definecolor{codegray}{rgb}{0.5,0.5,0.5}
\definecolor{codepurple}{rgb}{0.58,0,0.82}
\definecolor{backcolour}{rgb}{0.95,0.95,0.92}

\lstdefinestyle{mystyle}{
    backgroundcolor=\color{backcolour},   
    commentstyle=\color{codegreen},
    keywordstyle=\color{magenta},
    numberstyle=\tiny\color{codegray},
    stringstyle=\color{codepurple},
    basicstyle=\ttfamily\small,
    breakatwhitespace=false,         
    breaklines=true,                 
    captionpos=b,                    
    keepspaces=true,                 
    numbers=left,                    
    numbersep=5pt,                  
    showspaces=false,                
    showstringspaces=false,
    showtabs=false,                  
    tabsize=2
}
\newtheorem{thm}{Theorem}[section]
\newtheorem{lem}[thm]{Lemma}
\newtheorem{prop}[thm]{Proposition}
\newtheorem{cor}[thm]{Corollary}
\theoremstyle{definition}
\newtheorem{defn}[thm]{Definition}
\newtheorem{exa}[thm]{Example}
\newtheorem{rem}[thm]{Remark}

\usepackage{hyperref}
\usepackage{cleveref}

\crefname{thm}{theorem}{theorems}
\Crefname{thm}{Theorem}{Theorems}
\crefname{lem}{lemma}{lemmas}
\Crefname{lem}{Lemma}{Lemmas}
\crefname{prop}{proposition}{propositions}
\Crefname{prop}{Proposition}{Propositions}
\crefname{cor}{corollary}{corollaries}
\Crefname{cor}{Corollary}{Corollaries}
\crefname{defn}{definition}{definitions}
\Crefname{defn}{Definition}{Definitions}
\crefname{exa}{example}{examples}
\Crefname{exa}{Example}{Examples}
\crefname{rem}{remark}{remarks}
\Crefname{rem}{Remark}{Remarks}
\crefname{axi}{axiom}{axioms}
\Crefname{axi}{Axiom}{Axioms}
\crefname{pro}{property}{properties}
\Crefname{pro}{Property}{Properties}
\crefname{conj}{conjecture}{conjectures}
\Crefname{conj}{Conjecture}{Conjectures}

\AddToHook{env/thm/begin}{\crefalias{thm}{thm}}
\AddToHook{env/lem/begin}{\crefalias{thm}{lem}}
\AddToHook{env/prop/begin}{\crefalias{thm}{prop}}
\AddToHook{env/cor/begin}{\crefalias{thm}{cor}}
\AddToHook{env/defn/begin}{\crefalias{thm}{defn}}
\AddToHook{env/exa/begin}{\crefalias{thm}{exa}}
\AddToHook{env/rem/begin}{\crefalias{thm}{rem}}
\AddToHook{env/axi/begin}{\crefalias{thm}{axi}}
\AddToHook{env/pro/begin}{\crefalias{thm}{pro}}
\AddToHook{env/conj/begin}{\crefalias{thm}{conj}}

\DeclareMathOperator{\dis}{d}
\DeclareMathOperator{\diss}{d_{\mathrm{st}}}
\DeclareMathOperator{\sign}{sign}
\DeclareMathOperator{\St}{St}

\DeclareMathOperator{\Der}{Der}
\DeclareMathOperator{\val}{val}

\newcommand{\Rzl}{\mathbb{R}^{\mathbb{Z}_{<}}}
\newcommand{\Rzlc}{\mathbb{R}_c^{\mathbb{Z}_{<}}}

\newcommand{\einf}{\mathfrak{e}}
\newcommand{\1}{\text{\ding{172}}}

\newcommand{\N}{\mathbb{N}}

\newcommand*{\underdownarrow}[1]{\ensuremath{\underset{\downarrow}{#1}}}

\title[Na\"ive Infinitesimal Analysis]{A Constructive Field of Infinitesimals:\\Chunk and Permeate Approach}

\author[A S Nugraha]{Anggha S. Nugraha}
\givenname{Anggha}
\surname{Nugraha}
\address{Institute of Computer Science\\
Czech Academy of Sciences\\
Prague, Czech Republic
}
\email{nugraha@cs.cas.cs}
\urladdr{https://stefannug.github.io/}

\keywords{Na\"ive infinitesimal analysis, Paraconsistent strategy, Grossone, Computability}
\subject{primary}{msc2020}{03H05, 26E35, 03B53}
\subject{secondary}{msc2020}{68Q17, 03F60, 54E35}
\arxivreference{arXiv:2607.08206v2}
\arxivpassword{}

\volumenumber{}
\issuenumber{}
\publicationyear{}
\papernumber{}
\startpage{}
\endpage{}
\doi{}
\MR{}
\Zbl{}
\received{}
\revised{}
\accepted{}
\published{}
\publishedonline{}
\proposed{}
\seconded{}
\corresponding{}
\editor{}
\version{}
\begin{document}

\begin{abstract}
While intuitive, na\"ive infinitesimal reasoning is classically inconsistent, and rigorous nonstandard analysis relies on non-constructive machinery. We resolve this tension by constructing an explicit, totally ordered field $\Rzl$ using only real sequences and Cauchy convolution. We model the combined real and hyperreal axioms via the Chunk and Permeate strategy, a paraconsistent technique that isolates contradictions without global collapse. Equipping $\Rzl$ with a two-tier topology, we develop a calculus where infinitesimal derivatives and integrals permeate cleanly to their classical counterparts. We further introduce a $(k,n)$-continuity hierarchy capturing infinitesimal smoothness invisible to standard or transfer-based models. Finally, $\Rzl$ yields a direct algebraic consistency proof for Sergeyev's Grossone arithmetic, and we establish strict computability bounds on field operations. By guaranteeing infinitesimal contradictions never reach the classical chunk, this work bridges paraconsistent logic, constructive mathematics, and nonstandard analysis into a transparent, computationally tractable framework for infinitesimal reasoning.
\end{abstract}

\maketitle

%%%%%%%%%%%%%%%%%%%% Start of main body of article

\vspace{-0.5cm}
\section{Introduction and Motivation}
\label{sec:intro}
\vspace{-0.5cm}
Since the inception of calculus, infinitesimals have been a source of both power and paradox. Newton and Leibniz freely used them, treating an infinitesimal \(h\) as nonzero when dividing by it, yet zero when discarding it at the end. The classic computation of the derivative of \(f(x)=ax^{2}+bx+c\) illustrates this:
\begin{align*}
f'(x) &= \frac{f(x+h)-f(x)}{h} 
= 2ax+h+b = 2ax+b.
\end{align*}
The inconsistency lies in the treatment of \(h\): it is nonzero in the denominator (otherwise the quotient is undefined) but is simply set to zero in the final step. This logical tension was formalised in the 19$^\text{th}$ century by the \(\epsilon\)-\(\delta\) limit, which eliminated infinitesimals from rigorous analysis. Nevertheless, infinitesimals remain intuitively appealing and are still used in physics and engineering \cite{susskind2014theoretical}.

In classical logic, any contradiction leads to explosion (\emph{ex contradictione quodlibet}), making inconsistent theories trivial. Paraconsistent logics, which reject explosion, allow us to reason with local inconsistencies without global collapse \cite{priest2002paraconsistent}. some works in inconsistent mathematics \cite{mortensen1995inconsistent,weber2010extensionality,mckubre2012real} has shown that one can develop real analysis, set theory, and geometry in paraconsistent settings. A particularly relevant precursor is the work on applying paraconsistent logic to real analysis and set theory \cite{mckubre2012real, weber2010extensionality}. Our paper continues this line by showing how a concrete field model can be obtained from contradictory axioms using the Chunk and Permeate strategy, a methodology already successfully employed in paraconsistent infinitesimal calculus \cite{brown2004chunk}.

Meanwhile, nonstandard analysis \cite{robinson1974non,goldblatt1998lectures} rigorously introduces infinitesimals via model theory and the transfer principle. The transfer principle, while elegant, relies on non-constructive tools (ultrafilters, choice) and is computationally opaque \cite{tao2012acheap,kanovei2013nonstandard}. Simpler, more elementary approaches have been sought: Tao’s cheap version of nonstandard analysis \cite{tao2012acheap} and the Alpha-theory of Benci and Di~Nasso \cite{benci2022measure} both avoid the full ultrapower construction, yet they still require some non-constructive choice (a free ultrafilter or a particular ideal). Our approach goes a step further. By explicitly managing the inconsistency with the Chunk and Permeate strategy, we obtain a fully constructive field model that needs no choice principles at all.

An explicit field constructed from real sequences offers full transparency, i.e. every object is a concrete sequence, every operation is performed componentwise or via convolution, and every property can be verified without appealing to set-theoretic principles beyond ordinary real analysis. This constructive approach not only yields a computationally tractable model but also reveals a rich hierarchy of continuity notions, such as \((k,n)\)-continuity, that are invisible under the classical continuum or under transfer-based hyperreals.

In this paper we propose a new number system \(\Rzl\) built explicitly from real sequences with a finite tail of infinities and a countable tail of infinitesimals. We obtain \(\Rzl\) by throwing together the languages of \(\R\) and \(^*\R\), i.e. combining their axioms, and then using the \emph{Chunk and Permeate} (C\&P) reasoning strategy \cite{brown2004chunk} to manage the resulting inconsistencies. The C\&P approach partitions the theory into consistent chunks and allows selective information to flow between them, yielding a concrete, computable model of an infinitesimal-enriched field without ultrafilters or heavy model theory.

Our main contributions are: a direct construction of \(\Rzl\) as a totally ordered field (\Cref{sec:construction}); a natural two-tier topology (\Cref{sec:topology}); a calculus including derivatives, continuity, convergence, and Riemann integration within the C\&P permeability scheme (\Cref{sec:calculus}--\Cref{sec:integration}); a fine-grained hierarchy of \((k,n)\)-continuity (\Cref{sec:continuity}); a computability analysis (\Cref{sec:computability}); and a direct consistency proof for Grossone arithmetic, which has found recent applications in infinitesimal probabilities \cite{calude2020infinitesimal} and in infinitesimally punctured physical models \cite{smarandache2026infinitesimal}. This work contributes to the dialogue between constructive and nonstandard analysis \cite{sanders2013connection,sanders2015effective} and opens avenues for applications in physics and reverse mathematics.

\vspace{-0.5cm}
\section{Background: From Transfer Principle to Chunk and Permeate}
\label{sec:prelim}

% \vspace{-0.5cm}
\subsection{The Union of \texorpdfstring{$\R$}{R} and \texorpdfstring{$^*\R$}{*R} is Inconsistent}

We work with a first-order language $\mathfrak{L}$ containing symbols for every real constant, every real-valued function, and every relation on $\R$. The real line $\R$ is the standard model of $\mathfrak{L}$. The hyperreal system $^*\R$ is another model, obtained as an ultrapower of $\R$ modulo a free ultrafilter on $\N$, which extends $\R$ with infinitesimal and infinite elements. Relations and functions are extended pointwise,
% via the ultrafilter
and the \emph{transfer principle} guarantees that a first-order sentence holds in $\R$ iff it holds in $^*\R$.

The transfer principle, however, has serious drawbacks, e.g. it is non-constructive (relying on the Axiom of Choice) and computationally opaque. Moreover, the cost-benefit analysis is poor. One must erect a substantial model-theoretic apparatus to transfer relatively few theorems. These issues have led to a search for alternative, more elementary treatments of infinitesimals.

A natural attempt to avoid heavy machinery is to simply take the union of the axioms of $\R$ and those of $^*\R$.  
Call this combined theory $T_{\mathrm{comb}}$.  
$T_{\mathrm{comb}}$ is plainly inconsistent: it asserts the existence of a positive infinitesimal (from $^*\R$) and the Archimedean property (from $\R$), which cannot hold together.  
In classical logic, an inconsistency would trivialise the theory, allowing any sentence to be proved (\emph{ex contradictione quodlibet}).  
To extract meaningful mathematics from $T_{\mathrm{comb}}$ we need a logical strategy that tolerates local contradictions without global collapse.

\vspace{-0.5cm}
\subsection{Chunk and Permeate: A Paraconsistent Inference Strategy}

The \emph{Chunk and Permeate} (C\&P) strategy \cite{brown2004chunk} is a reasoning method rooted in paraconsistent logic. A logic is \emph{paraconsistent} if the explosion rule $A,\neg A \vdash B$ is not generally valid. The C\&P approach works by splitting an inconsistent theory into several consistent \emph{chunks}, defining a \emph{permeability relation} that allows certain formulas to pass from one chunk to another, and reasoning with classical logic \emph{inside} each chunk.  
This prevents global explosion while allowing useful information to cross between chunks.

In our setting we divide the axioms of $T_{\mathrm{comb}}$ into two classical first-order theories:
\begin{itemize}[nosep]
  \item \textbf{Source chunk} $\Sigma_S$: contains the field axioms, order axioms, and the existence of a positive infinitesimal (a non-Archimedean ordered field).  
  \item \textbf{Target chunk} $\Sigma_T$: contains the field axioms, order axioms, completeness (every non-empty bounded set has a supremum), and the Archimedean property (i.e. the axioms of a Dedekind-complete ordered field, which characterise $\R$ up to isomorphism).
\end{itemize}

\begin{figure}[ht]
    \centering
    \begin{tikzpicture}[
        chunk/.style={
            draw,
            thick,
            minimum width=2cm,
            minimum height=0.2cm,
            align=center
        },
        arrow/.style={
            ->,
            >=stealth,
            thick
        }
    ]

        % Source chunk
        \node[chunk, font=\footnotesize] (S) at (0,0)
            {Source chunk $\Sigma_S$\\
            (non-Archimedean field $\Rzl$)};

        % Target chunk -- moved farther right
        \node[chunk, font=\footnotesize] (T) at (7,0)
            {Target chunk $\Sigma_T$\\
            (complete ordered field $\R$)};

        % Arrow
        \draw[arrow] (S.east) -- (T.west);

        % Labels placed explicitly in the gap
        \node[align=center, font=\footnotesize] at (3.6,0.25)
            {permeation $\rho$};

        \node[align=center, font=\footnotesize] at (3.6,-0.25)
            {only eqs. $\St(E)=a$};

        % Explanation below
        \node[align=center, font=\footnotesize] at (3.4,-1)
            {Infinitesimal statements
            that do not permeate\\
            remain confined to $\Sigma_S$};
    \end{tikzpicture}
    \caption{The Chunk \& Permeate scheme for infinitesimal calculus.}
    \label{fig:chunkpermeate}
\end{figure}
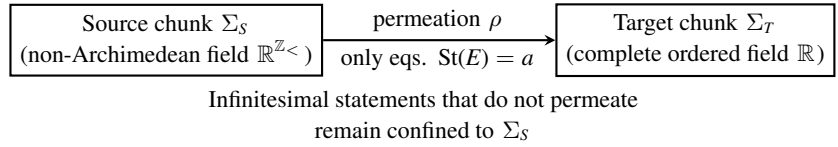

Let $\mathrm{Sent}(\Sigma_S)$ and $\mathrm{Sent}(\Sigma_T)$ be the sets of sentences of the respective chunks. We then specify a \emph{permeability relation} $\rho \subseteq \mathrm{Sent}(\Sigma_S) \times \mathrm{Sent}(\Sigma_T)$.  
Intuitively, $(\phi,\psi)\in\rho$ means that the `real shadow’ $\psi$ can be inferred from the `infinitesimal statement’ $\phi$.  
The C\&P rule sanctions the deduction
\[
\infer[\text{C\&P}]{\Sigma_T \vdash \psi}{
  \Sigma_S \vdash \phi & (\phi,\psi)\in\rho
}
\]
while \emph{no} converse permeation is allowed. Contradictions that involve only formulas that do not permeate (e.g. the existence of infinitesimals) remain confined to $\Sigma_S$ and never reach $\Sigma_T$. In particular, $T_{\mathrm{comb}}$ is \emph{non-trivial} under this strategy: although $\Sigma_S\cup\Sigma_T$ is inconsistent in classical logic, the target chunk $\Sigma_T$ never proves a contradiction because the permeability relation filters out the offending statements.

\vspace{-0.4cm}
\subsection{Modelling the chunks and verifying non-triviality}
We now show that both chunks are consistent and that the C\&P strategy indeed preserves the non-triviality of the combined system.
\begin{itemize}[nosep]
  \item \textbf{Target chunk $\Sigma_T$} is modelled by the real numbers $\R$ – a Dedekind-complete ordered field.
  \item \textbf{Source chunk $\Sigma_S$} is modelled by the field $\Rzl$ that we construct explicitly in Section~\ref{sec:construction} from real sequences with finite negative support and Cauchy convolution.  $\Rzl$ is a totally ordered field that contains a positive infinitesimal $\epsilon$ and an infinite element $\omega = 1/\epsilon$, thus satisfying $\Sigma_S$.
\end{itemize}

Because $\Sigma_S$ and $\Sigma_T$ are modelled by set-theoretic structures, each chunk is consistent (relative to the consistency of ordinary real analysis).  
The combined theory $T_{\mathrm{comb}}$ would be inconsistent if all its axioms were used freely. However, under the C\&P discipline only those consequences of $\Sigma_S$ that stand in the relation $\rho$ can be exported to $\Sigma_T$. In our concrete realisation, $\rho$ will only allow equations that involve the \emph{standard part map} $\St$ to pass from $\Sigma_S$ to $\Sigma_T$. For instance, the infinitesimal statement $\St(\Der_f(\mathbf{x},\epsilon)) = g(x)$ permeates to the classical derivative statement $f'(x)=g(x)$, but the statement ``there exists a positive infinitesimal'' does not permeate (it is not in the domain of $\St$ in the required way). Consequently, $\Sigma_T$ never acquires the ability to derive a contradiction from the infinitesimal existence axiom.

This demonstrates that $T_{\mathrm{comb}}$ is non-trivial in the sense of C\&P, and the construction provides a consistent model of the permeated theorems of classical analysis inside the real numbers, exactly mirroring the Newton-Leibniz `infinitesimal quotient $\to$ classical derivative' move.

\vspace{-0.5cm}
\section{Construction of the New Number System}
\label{sec:construction}
\subsection{The Field \texorpdfstring{$\Rzl$}{Rzl} as Formal Laurent Series}

Define $\Rzl$ as the set of all two-sided sequences of real numbers with only finitely many nonzero entries on the negative (infinite) side:
\[
\Rzl = \left\{\mathbf{x} = \langle\dots,\,x_{-2},\,x_{-1},\,\widehat{x_0},\,x_1,\,x_2,\dots\rangle \mid \exists N\ \forall n<-N\ x_n=0\right\}.
\]
The hat marks the \emph{standard part} $\St(\mathbf{x}) = x_0$. This identification is the key that will later allow us to permeate information from the source chunk to the target chunk. Equivalently, each element can be written as a formal Laurent series $\mathbf{x} = \sum_{n=-\infty}^\infty x_n \epsilon^n$,
% \[
% \mathbf{x} = \sum_{n=-\infty}^\infty x_n \epsilon^n,
% \]
where $\epsilon = \langle\widehat{0},1,0,\dots\rangle$ is a fixed positive infinitesimal, and its reciprocal $\omega = \epsilon^{-1} = \langle1,\widehat{0},0,\dots\rangle$ is infinite. Only finitely many negative powers appear.

Addition is componentwise, and multiplication is the Cauchy convolution:
\[
(\mathbf{x}\cdot\mathbf{y})_n = \sum_{k\in\Z} x_k\, y_{n-k},
\]
a finite sum. This makes $\Rzl$ a commutative ring with identity $\mathbf{1} = \langle\widehat{1},0,0,\dots\rangle$. To become a field, we need every nonzero element to be invertible, which is shown below.

\begin{lem}[Inverses]\label{lem:inverse}
Every nonzero $\mathbf{x}\in\Rzl$ has a unique multiplicative inverse.
\end{lem}
\vspace{-0.5cm}
\begin{proof}
Let $m = \min\{n\in\Z : x_n \neq 0\}$ be the valuation of $\mathbf{x}$. Write $\mathbf{x} = \epsilon^{m} u$ with $u_0 \neq 0$. We construct $v = \sum_{n\ge 0} v_n \epsilon^n$ satisfying $u v = 1$. The convolution gives $u_0 v_0 = 1$, and for $n\ge 1$, $\sum_{j=0}^{n} u_j v_{n-j} = 0$ yields the recurrence $v_n = -\frac{1}{u_0}\sum_{j=1}^{n} u_j v_{n-j}$. This determines all $v_n$ uniquely. Then $\mathbf{x}^{-1} = \epsilon^{-m} v$ is well-defined and has a finite negative part. Hence $\Rzl$ is a field.
\end{proof}

The \emph{valuation} $\val(\mathbf{x}) = \min\{n : x_n \neq 0\}$ (with $\val(0)=\infty$) is discrete and provides a non-Archimedean absolute value. The order is lexicographic from the most negative index: $\mathbf{x} < \mathbf{y}$ iff $\val(\mathbf{y}-\mathbf{x}) = n$ and $y_n - x_n > 0$. This makes $\Rzl$ a totally ordered field. The embedding $\R\hookrightarrow\Rzl$ given by $r \mapsto \langle\widehat{r},0,0,\dots\rangle$ preserves the order. Under this order, $\epsilon$ is a positive infinitesimal ($0<\epsilon<1/n$ for all $n\in\N$), and $\omega$ is infinite.

\begin{defn}
An element $\mathbf{x}\in\Rzl$ is \emph{infinitesimal} if $|\mathbf{x}| < 1/n$ for every $n\in\N$, \emph{finite} if $|\mathbf{x}| < r$ for some $r\in\R$, \emph{infinite} if $|\mathbf{x}| > r$ for all $r\in\R$, and \emph{appreciable} if it is finite but not infinitesimal.
% \begin{itemize}
% \item \emph{infinitesimal} if $|\mathbf{x}| < 1/n$ for every $n\in\N$;
% \item \emph{finite} if $|\mathbf{x}| < r$ for some $r\in\R$;
% \item \emph{infinite} if $|\mathbf{x}| > r$ for all $r\in\R$;
% \item \emph{appreciable} if it is finite but not infinitesimal.
% \end{itemize}
\end{defn}

\begin{lem}[Standard part homomorphism]
\label{lem:Standardparthomomorphism}
The map $\St : \mathcal{F} \to \R$, where $\mathcal{F} = \{\mathbf{x}\in\Rzl : \mathbf{x} \text{ is finite}\}$, is a surjective ring homomorphism whose kernel is exactly the set of infinitesimals.
\end{lem}
\vspace{-0.5cm}
\begin{proof}
For any finite $\mathbf{x},\mathbf{y}$, the components satisfy $(\mathbf{x}+\mathbf{y})_0 = x_0+y_0$ and $(\mathbf{x}\cdot\mathbf{y})_0 = x_0y_0$. Therefore $\St(\mathbf{x}+\mathbf{y}) = \St(\mathbf{x})+\St(\mathbf{y})$ and $\St(\mathbf{x}\cdot\mathbf{y}) = \St(\mathbf{x})\St(\mathbf{y})$. Surjectivity is clear, and $\St(\mathbf{x})=0$ iff $x_0=0$, which by definition means $\mathbf{x}$ is infinitesimal. Hence the kernel is the ideal of infinitesimals.
\end{proof}

The algebraic fact above is crucial for the permeation of derivatives and integrals later on. For $m\in\N$, set $\Delta^m = \{a\epsilon^m : a\in\R\}$ and $\Delta^{\underdownarrow{m}} = \bigcup_{n\ge m}\Delta^n$. These subsets stratify the infinitesimals into levels of magnitude, a structure that will be essential for the $(k,n)$-continuity hierarchy.

\vspace{-0.5cm}
\subsection{Grossone Arithmetic in \texorpdfstring{$\Rzl$}{Rzl}}
\vspace{-0.3cm}
Sergeyev’s Grossone $\1$ is an infinite number designed to behave like an ordinary natural number, so that standard arithmetic rules apply seamlessly to infinite and infinitesimal quantities \cite{sergeyev2013arithmetic}. 
In our field $\Rzl$ we simply set $\1 = \omega = \epsilon^{-1}$. 
Then all axioms of Grossone (e.g.\ $\1$ larger than any finite natural number, $\1-1<\1$, $\1/\1=1$) are satisfied as immediate field-theoretic facts. 
For example, $\1-1 = \omega-1$ has valuation $-1$ and its most significant nonzero coefficient (at index $0$) is $-1$, while $\omega$ has first coefficient $1$ at index $-1$. Since $0 > -1$ in the lexicographic order, we obtain $\omega-1 < \omega$, i.e.\ $\1-1 < \1$. Moreover, complicated expressions like $(\1 + 1/\1)^{-1}$ can be expanded explicitly as series, providing concrete computational content.

This construction yields a \emph{direct, finitary consistency proof} for the Grossone axioms, entirely within the usual real numbers plus sequences and convolution. 
Unlike the meta-mathematical proof of Lolli \cite{lolli2015metamathematical}, which requires coding the syntax of Grossone’s theory and a model-theoretic argument, our model is purely algebraic and fully explicit, i.e. every element is a concrete sequence, every operation is computable, and the axioms are verified by elementary calculations. Thus the consistency of Grossone arithmetic is reduced to the consistency of ordinary real analysis, with no hidden set-theoretic assumptions or non-constructive steps.

\vspace{-0.5cm}
\section{Topology on \texorpdfstring{$\Rzl$}{Rzl}}
\label{sec:topology}

\subsection{Metrics, balls, and topological properties}

The valuation $\val(\mathbf{x}) = \min\{n : x_n \neq 0\}$ (with
$\val(0)=\infty$) induces an ultrametric
\[
\dis(\mathbf{x},\mathbf{y}) = 2^{-\val(\mathbf{x}-\mathbf{y})}
\qquad (\text{with } 2^{-\infty}=0),
\]
generating the \emph{valuation topology} $\tau_v$.  This topology
distinguishes infinitesimally close points and is the natural setting
for the $(k,n)$-continuity hierarchy and the Newton quotient.

Let $\mathcal{F} = \{\mathbf{x}\in\Rzl : \mathbf{x}\text{ is finite}\}$
be the ring of finite elements.  The standard part map
$\St : \mathcal{F} \to \R$ is a surjective ring homomorphism (\Cref{lem:Standardparthomomorphism}).  Define a pseudo-metric on $\mathcal{F}$ by
\[
\diss(\mathbf{x},\mathbf{y}) = |\St(\mathbf{x}) - \St(\mathbf{y})|,
\]
which generates the \emph{standard topology} $\tau_{\St}$ on
$\mathcal{F}$.  This topology identifies points with the same standard part, mirroring classical real analysis. Infinite elements are excluded because $\St$ is not defined there, all constructions of
calculus (e.g. derivatives, integrals, continuity relative to $\tau_{\St}$) take place inside $\mathcal{F}$ anyway.)

For a positive real number $r$, an St-ball of radius $r$ around
$\mathbf{a}\in\mathcal{F}$ is
\[
B_{\St}(\mathbf{a},r) = \{\mathbf{x}\in\mathcal{F} :
\diss(\mathbf{x},\mathbf{a}) < r\}.
\]
For the valuation topology, a v-ball of radius $r = 2^{-k}$ (or any
positive element) around $\mathbf{a}\in\Rzl$ is
\[
B_v(\mathbf{a},r) = \{\mathbf{x}\in\Rzl : \dis(\mathbf{x},\mathbf{a}) < r\}.
\]

% \begin{table}[h]
% \centering
% \caption{Metrics and balls used in the paper}
% \label{table:balls}
% \begin{tabular}{lll}
% \toprule
% Set & Metric & Typical ball \\
% \midrule
% $\R$ & $\rho(x,y)=|x-y|$ & $(x-r,\,x+r)$ \\
% $\Rzl$ & $\dis(\mathbf{x},\mathbf{y}) = 2^{-\val(\mathbf{x}-\mathbf{y})}$
%       & $B_v(\mathbf{x},r)$, clopen \\
% $\mathcal{F}$ & $\diss(\mathbf{x},\mathbf{y}) = |\St(\mathbf{x})-\St(\mathbf{y})|$
%               & $B_{\St}(\mathbf{x},r) = \St^{-1}\big((\St(\mathbf{x})-r,\,
%                 \St(\mathbf{x})+r)\big)$ \\
% \bottomrule
% \end{tabular}
% \end{table}

\begin{thm}[Topological properties]
\label{thm:topology}
\textcolor{white}{}
\begin{enumerate}[label=(\roman*)]
\item $(\mathcal{F},\tau_{\St})$ is not Hausdorff but is preregular.
\item $(\Rzl,\tau_v)$ is Hausdorff, first-countable, zero-dimensional,
      spherically complete, not locally compact, and not
      second-countable.
\end{enumerate}
\end{thm}
\vspace{-0.5cm}
\begin{proof}
(i) If $\St(\mathbf{x})=\St(\mathbf{y})$, every St-open set containing
one contains the other, so they cannot be separated.  If
$\St(\mathbf{x})\neq\St(\mathbf{y})$, the St-balls of radius
$\frac12|\St(\mathbf{x})-\St(\mathbf{y})|$ are disjoint, giving
preregularity.

(ii) Hausdorffness follows from the ultrametric.  The countable family
$\{B_v(\mathbf{x},2^{-k})\}$ is a local base, so the space is
first-countable.  Ultrametric balls are clopen, hence
zero-dimensional.  Spherical completeness follows from the field of formal Laurent series over the spherically complete field $\R$ (see e.g.~\cite{kanovei2013nonstandard}).  Non-second-countability follows from the uncountably many disjoint v-balls of radius $\epsilon$.
Non-connectedness is witnessed by the partition into
$S_1 = \{\mathbf{x}\mid \mathbf{x}\le0\text{ or }
\mathbf{x}\in\Delta^{\underdownarrow{1}}\}$ and
$S_2 = \{\mathbf{x}\mid \mathbf{x}>0,\ \mathbf{x}\notin\Delta^{\underdownarrow{1}}\}$,
both v-open.
\end{proof}

The interplay of the two topologies is essential.  The $\tau_v$
topology provides a fine-grained analysis of infinitesimal behaviour,
while $\tau_{\St}$ recovers classical real analysis on the finite
numbers by collapsing infinitesimals.  This dual perspective will be
exploited in the calculus sections that follow.

\vspace{-0.5cm}
\section{Permeation of Calculus in the C\&P Framework}
\label{sec:permeation-calculus}
\vspace{-0.5cm}
The C\&P strategy described in Section~\ref{sec:prelim} partitions the inconsistent combined theory $T_{\mathrm{comb}}$ into a source chunk $\Sigma_S$ (the field $\Rzl$) and a target chunk $\Sigma_T$ (the real numbers $\R$).  
A permeability relation $\rho \subseteq \mathrm{Sent}(\Sigma_S) \times \mathrm{Sent}(\Sigma_T)$ controls exactly which statements can travel from the infinitesimal world to the classical world.  
In this section we make $\rho$ explicit for calculus, explain why it covers the usual derivative and integral theorems, and highlight the logical limits that distinguish this approach from full transfer.

\vspace{-0.5cm}
\subsection{Microstable functions}
\vspace{-0.4cm}
For the standard part map to mediate between the two chunks, the functions under consideration must be insensitive to infinitesimal perturbations on the macroscopic level.

\begin{defn}[Microstability]\label{def:microstability}
A function $f:\mathcal{F}\to\Rzl$ is \emph{microstable} if $f(\mathbf{x})$ is finite for every finite $\mathbf{x}$, and for every finite $\mathbf{x}$ and every infinitesimal $\mathbf{\eta}$ (i.e.\ $\val(\mathbf{\eta})>0$),
\[
\St(f(\mathbf{x}+\mathbf{\eta})) = \St(f(\mathbf{x})).
\]
\end{defn}

The exponential, sine, and cosine can be extended to the finite elements of $\Rzl$ by Taylor expansion around their standard parts. The class of microstable functions is closed under the usual algebraic operations, guaranteeing that once we have a stock of elementary microstable functions, permeation extends automatically to all functions built from them.

\begin{lem}[Closure of microstability]
\label{lem:closure-microstable}
Suppose $f,g:\mathcal{F}\to\Rzl$ are microstable (hence $f(\mathbf{x})$, $g(\mathbf{x})$
are finite for all finite $\mathbf{x}$). Then the functions
$f+g$, $f\cdot g$, and $f\circ g$ (where defined) are also microstable.
\end{lem}
\vspace{-0.5cm}
\begin{proof}
For $f+g$ and $f\cdot g$ we use the homomorphism properties of $\St$ on
finite elements. Since $f(\mathbf{x})$ and $g(\mathbf{x})$ are finite,
$f(\mathbf{x})+g(\mathbf{x})$ and $f(\mathbf{x})\,g(\mathbf{x})$ are
finite. Moreover, for any infinitesimal $\mathbf{\eta}$,
\begin{align*}
\St((f+g)(\mathbf{x}+\mathbf{\eta}))
 &= \St(f(\mathbf{x}+\mathbf{\eta})) + \St(g(\mathbf{x}+\mathbf{\eta})) \\
 &= \St(f(\mathbf{x})) + \St(g(\mathbf{x})) = \St((f+g)(\mathbf{x})), \\
\St((f\cdot g)(\mathbf{x}+\mathbf{\eta}))
 &= \St(f(\mathbf{x}+\mathbf{\eta}))\,\St(g(\mathbf{x}+\mathbf{\eta})) \\
 &= \St(f(\mathbf{x}))\,\St(g(\mathbf{x})) = \St((f\cdot g)(\mathbf{x})).
\end{align*}
For composition, assume that $g$ maps finite numbers to finite numbers
(which is guaranteed by microstability). Then $f\circ g$ is defined on
$\mathcal{F}$ and yields finite values. Write
$g(\mathbf{x}+\mathbf{\eta}) = g(\mathbf{x}) + \mathbf{\delta}$ with
$\mathbf{\delta}$ infinitesimal (since $g$ is microstable and
$\mathbf{x}$ finite, $\St(g(\mathbf{x}+\mathbf{\eta})) =
\St(g(\mathbf{x}))$, so the difference is infinitesimal). Now
\begin{align*}
\St(f(g(\mathbf{x}+\mathbf{\eta})))
 &= \St(f(g(\mathbf{x}) + \mathbf{\delta})) \\
 &= \St(f(g(\mathbf{x}))) \quad \text{(by microstability of $f$)} \\
 &= \St((f\circ g)(\mathbf{x})).
\end{align*}
Thus $f\circ g$ is microstable.
\end{proof}

\vspace{-0.5cm}
\subsection{The general form of the permeability relation for calculus}
% \vspace{-0.5cm}
Every classical limit or integral arises as the standard part of a microstable infinitesimal expression.  
Accordingly, $\rho$ contains all pairs
\[
\bigl(\,\St(E) = a,\; E_{\R} = a\,\bigr),
\]
where $E$ is a term of $\Sigma_S$ built from microstable functions and $\epsilon$ (e.g. a Newton quotient or a Riemann sum), $a$ is a real number, and $E_{\R}$ is the corresponding classical real expression.  
Statements that assert the bare existence of infinitesimals, or that involve notions like compactness that are not visible in the valuation topology, never enter $\rho$.  
Hence contradictions that involve those notions remain confined to $\Sigma_S$, while the target chunk $\Sigma_T$ stays a faithful model of $\R$.

\vspace{-0.5cm}
\subsection{Permeation of derivatives and integrals}
\label{sec:permeation-der-int}
\vspace{-0.2cm}
As we will demonstrate in the following sections, both the classical derivative and the Riemann integral arise as standard parts of microstable infinitesimal expressions. By populating the permeability relation $\rho$ with pairs of the form $(\St(E) = a,\, E_{\R} = a)$, both pillars of elementary calculus permeate cleanly from $\Sigma_S$ to the target chunk $\Sigma_T$.

\vspace{-0.5cm}
\subsection{What does \emph{not} permeate?}
\vspace{-0.2cm}
The extreme value property (EVP) and the intermediate value property (IVP) illustrate the controlled nature of the C\&P strategy. In $\Sigma_T$ (the real numbers) every continuous function on $[a,b]$ satisfies EVP and IVP. In $\Sigma_S$ (the field $\Rzl$) these theorems fail for the valuation topology (see Section~\ref{sec:continuity}). Therefore $\Sigma_S$ can prove the \emph{negation} of EVP (or IVP) for a specific function, but that negation does not permeate to $\Sigma_T$, because it is not of the form $\St(E)=a$ with $E$ microstable. The C\&P rule is one-way: $\Sigma_S \vdash \phi$ and $(\phi,\psi)\in\rho$ yields $\Sigma_T \vdash \psi$, but no converse permeation is allowed. Thus the inconsistency between ``EVP holds'' and ``EVP fails'' is safely contained.

\vspace{-0.5cm}
\subsection{Comparison with the transfer principle}
\vspace{-0.2cm}
Traditional nonstandard analysis employs the \emph{transfer principle} to move \emph{all} first-order properties between $\R$ and $^*\R$, requiring a free ultrafilter and the Axiom of Choice. In contrast, the C\&P approach deliberately permits only a restricted family of equations (those involving standard parts of microstable expressions) to cross the chunk boundary. The gain is a fully constructive and explicit field model while the cost is that not every theorem of classical analysis is automatically available in the infinitesimal realm. One must verify microstability by hand, which, as we have shown, is straightforward for the elementary functions of calculus. Moreover, the C\&P framework is modular, i.e. if a new class of microstable expressions is identified, the permeability relation can be extended accordingly, mirroring the way mathematical analysis is built layer by layer.

\vspace{-0.5cm}
\section{Derivatives in \texorpdfstring{$\Rzl$}{Rzl}}
\label{sec:calculus}

\vspace{-0.3cm}
The valuation topology $\tau_v$ captures the infinitesimal behaviour needed to define Newton quotients, while the standard topology $\tau_{\St}$ (via the standard part map) connects the resulting expressions to classical real derivatives.  
Building on the permeability scheme of \Cref{sec:permeation-calculus}, we now develop the derivative calculus in detail.

% \subsection{Chunk \& Permeate Scheme for Derivatives}
% \label{sec:derivative-scheme}

% Recall from Definition~\ref{def:microstability} that a function $f:\Rzl\to\Rzl$ is \emph{microstable} if $\St(f(\mathbf{x}+\mathbf{\eta})) = \St(f(\mathbf{x}))$ for every infinitesimal $\mathbf{\eta}$.  
% As described in Section~\ref{sec:permeation-calculus}, the permeability relation $\rho$ allows equations of the form $\St(E)=a$ to pass from $\Sigma_S$ to $\Sigma_T$ whenever $E$ is built from microstable functions and the infinitesimal $\epsilon$.  
% For derivatives we instantiate this scheme with the Newton quotient.

% Within $\Sigma_S$ we define the difference quotient operator
% \[
% \Der_f(\mathbf{x},\mathbf{\epsilon}) = \frac{f(\mathbf{x}+\mathbf{\epsilon}) - f(\mathbf{x})}{\mathbf{\epsilon}},
% \qquad \mathbf{\epsilon}=\epsilon.
% \]

\vspace{-0.5cm}
\subsection{Chunk \& Permeate Scheme for Derivatives}
\label{sec:derivative-scheme}

To instantiate the permeation scheme for calculus, we define the difference quotient operator within $\Sigma_S$:
\[
\Der_f(\mathbf{x},\epsilon) = \frac{f(\mathbf{x}+\epsilon) - f(\mathbf{x})}{\epsilon}.
\]
If $f$ is the canonical extension of a real-analytic function, then \Cref{lem:permeation} guarantees that the standard part of this operator evaluates to the classical derivative. This allows us to map the infinitesimal quotient directly to the classical derivative via the permeability relation $\rho$.

\begin{defn}[Permeability relation for derivatives]
\label{def:permeability-deriv}
The permeability relation $\rho \subseteq \mathrm{Sent}(\Sigma_S) \times \mathrm{Sent}(\Sigma_T)$ is taken to contain all pairs of the form
\[
(\phi,\psi) \quad\text{where}\quad 
\phi \equiv \St(\Der_f(\mathbf{x},\epsilon)) = y,\;
\psi \equiv g'(x) = y,
\]
where $f$ is the canonical extension of the real-analytic function $g$, $x = \St(\mathbf{x})$, and $y \in \R$.
\end{defn}

\noindent\textbf{Permeation of derivatives for analytic extensions}
While microstability is a useful general notion, the standard-part evaluation of the Newton quotient requires a stronger condition, i.e. the function must be given locally by a convergent power series in the
infinitesimal part Let $g:\R\to\R$ be real-analytic.  Its \emph{canonical extension} to the finite elements of $\Rzl$ is defined by
\[
f(\mathbf{x}) \;=\; \sum_{k=0}^{\infty}
\frac{g^{(k)}(\St(\mathbf{x}))}{k!}\,
\bigl(\mathbf{x}-\St(\mathbf{x})\bigr)^{k},
\qquad \mathbf{x}\in\mathcal{F}.
\]
Because $\mathbf{x}-\St(\mathbf{x})$ has valuation $\ge 1$, the
$k$-th term has valuation $k$, so the series converges in the valuation
topology and yields a well-defined element $f(\mathbf{x})\in\Rzl$.  If
$\mathbf{x}$ is already a standard real $a$, the series reduces to
$g(a)$, so $f$ extends $g$.

\begin{lem}[Permeation of derivatives for analytic extensions]\label{lem:permeation}
Let $g:\R\to\R$ be real-analytic and $f$ its canonical extension as
above.  Then for every finite $\mathbf{x}$,
\[
\St\!\left(\frac{f(\mathbf{x}+\epsilon)-f(\mathbf{x})}{\epsilon}\right)
\;=\; g'\bigl(\St(\mathbf{x})\bigr).
\]
\end{lem}
\begin{proof}
Write $\mathbf{x}=a+\eta$ with $a=\St(\mathbf{x})$ and $\eta$
infinitesimal.  Then
\begin{align*}
f(\mathbf{x}+\epsilon) 
&= \sum_{k=0}^{\infty} \frac{g^{(k)}(a)}{k!}\,(\eta+\epsilon)^k \\
&= \sum_{k=0}^{\infty} \frac{g^{(k)}(a)}{k!}
     \sum_{j=0}^{k}\binom{k}{j}\eta^{\,k-j}\epsilon^{\,j}.
\end{align*}
The term $j=0$ gives $f(\mathbf{x})$.  Subtracting $f(\mathbf{x})$ and
dividing by $\epsilon$ leaves
\[
\frac{f(\mathbf{x}+\epsilon)-f(\mathbf{x})}{\epsilon}
 = \sum_{k=1}^{\infty} \frac{g^{(k)}(a)}{k!}
     \sum_{j=1}^{k}\binom{k}{j}\eta^{\,k-j}\epsilon^{\,j-1}.
\]
The inner sum starts at $j=1$, and for $j=1$ we obtain
\[
\frac{g^{(k)}(a)}{k!}\,k\,\eta^{\,k-1}
 = \frac{g^{(k)}(a)}{(k-1)!}\,\eta^{\,k-1}.
\]
Summing over $k\ge 1$ yields
\[
\sum_{m=0}^{\infty} \frac{g^{(m+1)}(a)}{m!}\,\eta^{\,m}
 = g'(a+\eta) \quad\text{(Taylor series of $g'$)}.
\]
All remaining terms ($j\ge 2$) contain at least one factor $\epsilon$,
so their valuation is $\ge 1$.  Hence
\[
\frac{f(\mathbf{x}+\epsilon)-f(\mathbf{x})}{\epsilon}
 = g'(a+\eta) + \delta,\qquad \val(\delta)\ge 1.
\]
Taking the standard part kills $\delta$ and gives
$\St(g'(a+\eta)) = g'(a)$, because $g'$ is continuous on $\R$ and
$\eta$ is infinitesimal.
\end{proof}

The restriction to analytic extensions is harmless for calculus. Polynomials, $\exp$, $\sin$, $\cos$, and all elementary transcendental functions are still covered, and the class is closed under addition, multiplication, and composition (the composition of canonical extensions of $g$ and $h$ is the canonical extension of $g\circ h$ whenever $h$ is real-analytic).  The earlier closure lemma for
microstable functions remains true but is now irrelevant for derivative permeation. It is kept for later use in the continuity and integration
sections.

% \subsection{Derivatives of Elementary and Singular Functions}
% For a polynomial $p(\mathbf{x}) = \sum_{k=0}^{n} a_k \mathbf{x}^{k}$, expanding gives
% \[
% \frac{p(\mathbf{x}+\epsilon)-p(\mathbf{x})}{\epsilon}
% = \sum_{k=1}^{n} a_k\Bigl(k\mathbf{x}^{k-1} + \binom{k}{2}\mathbf{x}^{k-2}\epsilon + \dots + \epsilon^{k-1}\Bigr),
% \]
% and taking the standard part eliminates all terms containing a positive power of $\epsilon$, leaving the classical derivative $\sum k a_k x^{k-1}$. The same formal manipulation applies to any convergent power series. Hence, the exponential, sine, and cosine functions, defined by their Maclaurin series, all satisfy the permeation lemma with their classical derivatives $e^{x}$, $\cos x$, and $-\sin x$ respectively. Explicit expansions confirm this. For $\mathbf{x} = x + a\epsilon$ one obtains, for instance,
% \[
% \sin(x+a\epsilon) = \langle \widehat{\sin x},\, a\cos x,\, -\tfrac{a^{2}}{2!}\sin x,\, \dots \rangle,
% \]
% and similarly for cosine and exponential.

% Even singular functions can be treated within this scheme without breakdown. The signum function $\sign(\mathbf{x}) = 1,0,-1$ according to whether $\St(\mathbf{x})>0, =0, <0$ is also microstable. Its difference quotient is identically zero because $\St(\mathbf{x}+\epsilon) = \St(\mathbf{x})$, so $\St(\Der_{\sign}) = 0$. This coincides with the distributional derivative of the Heaviside step function, showing that the C\&P framework naturally handles singularities.

\vspace{-0.5cm}
\subsection{Derivatives of Elementary and Singular Functions}
For a polynomial $p(\mathbf{x}) = \sum_{k=0}^{n} a_k \mathbf{x}^{k}$
restricted to the finite domain $\mathcal{F}$, expanding the difference
quotient gives
\[
\frac{p(\mathbf{x}+\epsilon)-p(\mathbf{x})}{\epsilon}
= \sum_{k=1}^{n} a_k\Bigl(k\mathbf{x}^{k-1}
   + \binom{k}{2}\mathbf{x}^{k-2}\epsilon + \dots + \epsilon^{k-1}\Bigr).
\]
Because $\mathbf{x}$ is finite, all terms $k\mathbf{x}^{k-1}$ evaluate
to finite elements.  Taking the standard part eliminates all terms
containing a positive power of $\epsilon$, leaving exactly the classical
derivative $\sum k a_k x^{k-1}$, where $x = \St(\mathbf{x})$.

The same formal manipulation applies to any convergent power series
evaluated on finite elements.  Hence the exponential, sine, and cosine
functions, defined by their Maclaurin series, are microstable and
satisfy the permeation lemma with their classical derivatives $e^{x}$,
$\cos x$, and $-\sin x$ respectively.  Explicit expansions confirm this.
For $\mathbf{x} = x + a\epsilon$ one obtains, for instance,
\[
\sin(x+a\epsilon) = \langle \widehat{\sin x},\, a\cos x,\,
-\tfrac{a^{2}}{2!}\sin x,\, \dots \rangle,
\]
and similarly for cosine and exponential.

Even singular functions can be treated within this scheme without
breakdown.  Consider the \emph{signum} function defined by
\[
\sign(\mathbf{x}) =
\begin{cases}
 1 & \text{if } \St(\mathbf{x}) > 0,\\
 0 & \text{if } \St(\mathbf{x}) = 0,\\
-1 & \text{if } \St(\mathbf{x}) < 0.
\end{cases}
\]
It is microstable (it depends only on the standard part) but is not
the canonical extension of a real-analytic function, so the permeation lemma does not apply. Direct evaluation of the Newton quotient gives
\[
\frac{\sign(\mathbf{x}+\epsilon)-\sign(\mathbf{x})}{\epsilon} = 0
\]
because $\sign(\mathbf{x}+\epsilon)=\sign(\mathbf{x})$ for all
$\mathbf{x}$ (the standard part of $\mathbf{x}+\epsilon$ equals the
standard part of $\mathbf{x}$).  Hence
$\St(\Der_{\sign}(\mathbf{x},\epsilon)) = 0$, which coincides with the
classical derivative of the signum function everywhere except at $0$, where the classical derivative does not exist but the standard part naturally assigns the value $0$.  This example shows that the framework can handle non-analytic functions, although the permeation
of derivatives is guaranteed only for analytic extensions.

\begin{rem}[Nuance of the signum function]
The definition of $\sign(\mathbf{x})$ based on $\St(\mathbf{x})$ is a deliberate choice that trades internal field exactness for microstability. In the totally ordered field $\Rzl$, the positive infinitesimal $\epsilon$ satisfies $\epsilon > 0$ strictly. If we had defined an \emph{internal} signum function based on the field's lexicographic order, it would evaluate to $1$ at $\epsilon$ and $0$ at $0$. However, this internal signum is not microstable because evaluating the Newton quotient at $0$ would yield $(1 - 0)/\epsilon = \omega$, an infinite element, causing permeation to fail. By defining the signum via the standard part, $\sign(\epsilon) = 0$, we ensure the function remains insensitive to infinitesimal perturbations (microstable), allowing its distributional derivative to permeate successfully.
\end{rem}

\vspace{-0.5cm}
\section{Continuity and Convergence in \texorpdfstring{$\Rzl$}{Rzl}}
\label{sec:continuity}

The two topologies also shed light on continuity and convergence.
The $(k,n)$-continuity hierarchy is defined purely in terms of the
valuation, i.e. within $\tau_v$, while the subsequent analysis of
convergence explicitly distinguishes valuational convergence (in
$\tau_v$) from standard convergence (in $\tau_{\St}$), thereby exploiting
the complementary roles of the two topologies.

\vspace{-0.5cm}
\subsection{\texorpdfstring{$(k,n)$}{kn}-continuity and Its Hierarchy}
The valuation on $\Rzl$ allows us to quantify how infinitesimal a change is, leading to a fine-grained refinement of classical continuity. The classical $\epsilon$-$\delta$ definition of continuity does not distinguish between different infinitesimal magnitudes. The valuation enables us to stratify continuity in a way that is invisible in $\R$. For $k,n\in\N$, a function $f$ is \emph{$(k,n)$-continuous} at $\mathbf{c}$ if
\[
\forall\mathbf{x}\; \bigl( \val(\mathbf{x}-\mathbf{c}) \ge n \implies \val(f(\mathbf{x})-f(\mathbf{c})) \ge k \bigr).
\]
In words, a perturbation of infinitesimal order at least $n$ in the input produces a change of infinitesimal order at least $k$ in the output.

\begin{thm}[Hierarchy of $(k,n)$-continuity]\label{thm:continuity-hierarchy}
For any $f:\Rzl\to\Rzl$:
\begin{enumerate}[label=(\roman*)]
\item If $f$ is $(k,n)$-continuous, then it is $(k,n+1)$-continuous.
\item If $f$ is $(k+1,n)$-continuous, then it is $(k,n)$-continuous.
\item If $f$ is $(k,n)$-continuous for all $k,n\in\N$, then $f$ is constant on every infinitesimal neighbourhood, i.e. for every $\mathbf{c}$ and every $\mathbf{x}$ with $\val(\mathbf{x}-\mathbf{c}) > 0$, $f(\mathbf{x}) = f(\mathbf{c})$.
\item (Composition) If $f$ is $(k,n)$-continuous and $g$ is $(n,q)$-continuous, then $f\circ g$ is $(k,q)$-continuous.
\item (Product) If $f$ is $(k,n)$-continuous, $g$ is $(l,o)$-continuous, and both are bounded (finite valuation), then $fg$ is $(\min(k,l),\max(n,o))$-continuous.
\end{enumerate}
\end{thm}
\vspace{-0.3cm}
\begin{proof}
(i) and (ii) are immediate. (iii) Fix $\mathbf{c}$ and let $\mathbf{x}$ satisfy $\val(\mathbf{x}-\mathbf{c}) = d > 0$. Since $f$ is $(k,d)$-continuous for every $k$, we have $\val(f(\mathbf{x})-f(\mathbf{c})) \ge k$ for all $k$, forcing $f(\mathbf{x}) = f(\mathbf{c})$. (iv) $\val(\mathbf{x}-\mathbf{c})\ge q \implies \val(g(\mathbf{x})-g(\mathbf{c}))\ge n \implies \val(f(g(\mathbf{x}))-f(g(\mathbf{c})))\ge k$. (v) Let $m = \max(n,o)$. If $\val(\mathbf{x}-\mathbf{c}) \ge m$, then
$\val(f(\mathbf{x})-f(\mathbf{c})) \ge k$ and
$\val(g(\mathbf{x})-g(\mathbf{c})) \ge l$.
Because $f$ and $g$ are bounded, their values are finite;
hence $\val(g(\mathbf{x}))\ge 0$ and $\val(f(\mathbf{c}))\ge 0$.
Using the ultrametric inequality,
\begin{align*}
\val(fg(\mathbf{x})-fg(\mathbf{c}))
&\ge
\min\bigl\{
\val(f(\mathbf{x})-f(\mathbf{c}))+\val(g(\mathbf{x})),\\
&\qquad
\val(g(\mathbf{x})-g(\mathbf{c}))+\val(f(\mathbf{c}))
\bigr\}
\ge \min\{k,l\}.
\end{align*}
Thus $fg$ is $(\min(k,l),m)$-continuous.
\end{proof}

\begin{rem}
The boundedness condition in property (v) is strictly necessary in $\Rzl$. In classical real analysis, the product of continuous functions is always continuous. However, evaluating at infinite elements breaks this. For example, the identity function $f(\mathbf{x}) = \mathbf{x}$ is $(1,1)$-continuous. Yet, at the infinite point $\mathbf{x} = \omega$, an infinitesimal perturbation $d\epsilon$ yields $f(\omega + d\epsilon)^2 = \omega^2 + 2d + d^2\epsilon^2$. The change in the output is $2d + d^2\epsilon^2$, which has valuation $0$ (appreciable). Thus, $f(\mathbf{x})^2$ fails to be $(1,1)$-continuous at infinity. The finite valuation restriction correctly isolates the continuous behavior.
\end{rem}

The following example illustrates how a function can fail to be $(0,0)$-continuous yet be $(0,1)$-continuous, a distinction that is completely missed by the classical definition.

% \begin{exa}
% \label{ex:xandx+1fork-ncontinuity}
% Consider the function
% \[
% f(\mathbf{x}) =
% \begin{cases} 
% \mathbf{x} & \text{if } \St(\mathbf{x}) \le 1, \\
% \mathbf{x} + \mathbf{1} & \text{otherwise}.
% \end{cases}
% \]
% At the standard point $1$, all numbers infinitesimally close to $1$ form an infinitesimal neighbourhood that when magnified, the function appears as a vertical segment (\Cref{fig:xandx+1fork-ncontinuity}).  
% It is not $(0,0)$-continuous (e.g. take $\einf = 1/2$ and points near $\mathbf{1.5}$),
% % It is not $(0,0)$-continuous (take $\einf_{1_0}=1/2$ near $\mathbf{1.5}$),
% but it is $(0,1)$-continuous: any two points whose difference is of order $\ge 1$ have the same standard part, hence lie in the same piece of the function. The identity function $g(\mathbf{x}) = \mathbf{x}$ is $(0,0)$-continuous but not $(1,0)$-continuous; in fact, $g$ is $(k,n)$-continuous iff $k \le n$.
% \end{exa}

\begin{exa}
\label{ex:kn-continuity-shift}
The $(k,n)$-continuity hierarchy allows us to classify functions by how they shift infinitesimal degrees, a distinction invisible in classical analysis. Consider the function $f(\mathbf{x}) = \omega \mathbf{x}$, which acts as an infinitesimal magnifier.
\begin{itemize}[nosep]
    \item $f$ is \textbf{$(0,1)$-continuous}: If the input is perturbed by a first-order infinitesimal, $\val(\mathbf{x}-\mathbf{y}) \ge 1$, the change in output is $\val(\omega(\mathbf{x}-\mathbf{y})) = -1 + \val(\mathbf{x}-\mathbf{y}) \ge 0$. Thus, a first-order infinitesimal change produces a finite (valuation $\ge 0$) change.
    \item $f$ is \textbf{not $(1,1)$-continuous}: Let $\mathbf{x} = \epsilon$ and $\mathbf{y} = 0$. The input change has valuation $1$, but the output change is $f(\epsilon) - f(0) = \omega\epsilon = \mathbf{1}$, which has valuation $0 \not\ge 1$.
    \item $f$ is \textbf{not $(0,0)$-continuous}: Let $\mathbf{x} = \mathbf{1}$ and $\mathbf{y} = 0$. The input change has valuation $0$, but the output change is $f(\mathbf{1}) - f(0) = \omega$, which has valuation $-1 \not\ge 0$.
\end{itemize}
Unlike the classical identity function (which is $(n,n)$-continuous for all $n$), $f(\mathbf{x}) = \omega \mathbf{x}$ uniformly degrades the continuity degree by one level.
\end{exa}

% \begin{figure}[ht]
%     \centering
%     \includegraphics[width=.3\linewidth]{Illustrationf_x_equaltoxforstd_x.jpg}
%     \caption{Illustration of $f(\mathbf{x})$ from Example \ref{ex:xandx+1fork-ncontinuity}}
%     \label{fig:xandx+1fork-ncontinuity}
% \end{figure}

\Cref{fig:Illustrationofeinf1andeinf2} provides a geometric illustration of the $(k,n)$-continuity condition.
\begin{figure}[ht]
    \centering
    \subfloat[An $\einf_1$ bound and its $\einf_2$ neighbourhood.]{%
        \includegraphics[width=.3\linewidth]{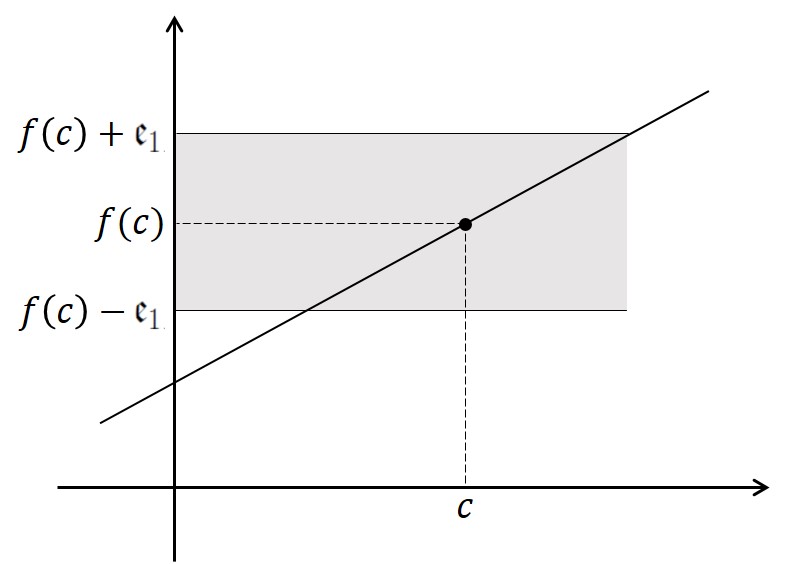}\hfill
        \includegraphics[width=.3\linewidth]{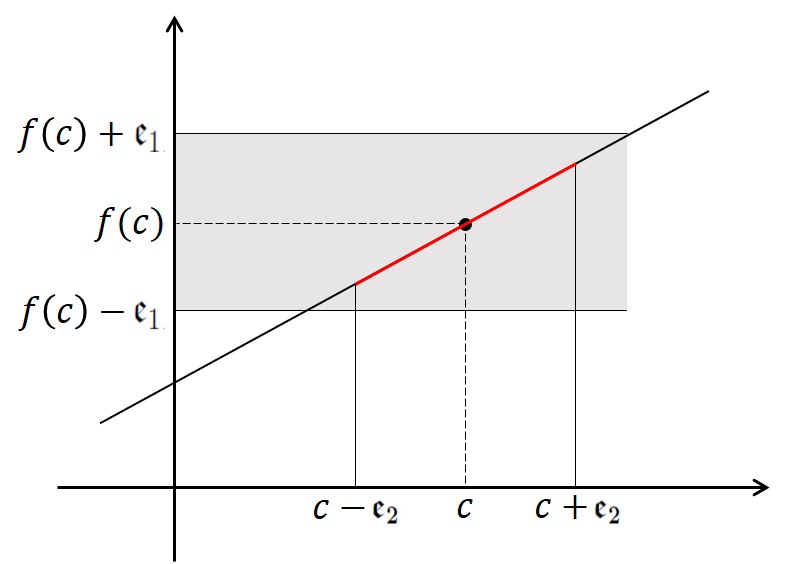}
    }
    \par\medskip
    \subfloat[A smaller bound and its finer neighbourhood.]{%
        \includegraphics[width=.4\linewidth]{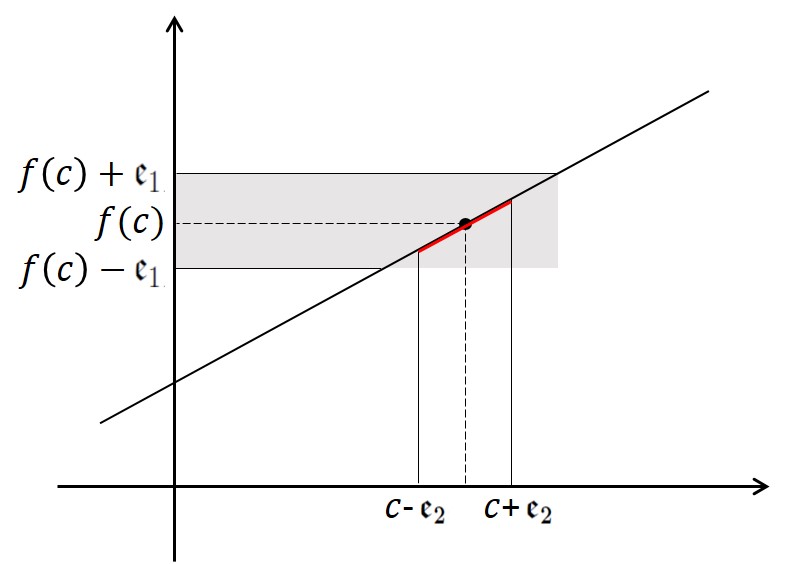}
    }
    \caption{Illustration of the $\einf_1$-$\einf_2$ condition for continuity.}
    \label{fig:Illustrationofeinf1andeinf2}
\end{figure}

\vspace{-0.5cm}
\subsection{Extreme Value, Intermediate Value, and Continuity}
\vspace{-0.3cm}
We recall the definitions of the extreme value property (EVP) and the intermediate value property (IVP) for functions with values in an ordered field. In classical real analysis, the $\epsilon$-$\delta$ definition of continuity (ED) implies both EVP and IVP. In $\Rzl$ the situation is different, i.e. the valuation topology is totally disconnected and closed bounded intervals are not compact, so the classical proofs do not carry over. For instance, the indicator of the infinitesimals,
$\mathbf{1}_{\Delta}$, is valuation-continuous but fails the
IVP (it takes only $0$ and $1$ on $[0,1]$). The EVP also fails in general: let $(q_n)_{n\in\N}$
be an enumeration of the rationals in $(0,1)$ and define
\[
f(\mathbf{x}) =
\begin{cases}
1 - \frac{1}{n+1}, & \text{if } \St(\mathbf{x}) = q_n,\\
0, & \text{otherwise}.
\end{cases}
\]
Since the sets $\{\mathbf{x} : \St(\mathbf{x}) = q_n\}$ are precisely the open balls of radius $1$ (clopen) in the valuation topology,
$f$ is locally constant and hence valuation-continuous on $[0,1]$.
Its supremum on $[0,1]$ is $1$, but $f(\mathbf{x})<1$ for all
$\mathbf{x}$, so the maximum is never attained.  Conversely,
$f(x)=1/x$ on $(0,1]$ satisfies IVP but not EVP. Hence ED, EVP, and IVP are pairwise independent in $\Rzl$. 

\vspace{-0.5cm}
\subsection{Convergence of Sequences}
We distinguish three modes of convergence for sequences in $\Rzl$, each corresponding to a different level of detail.

% \begin{defn}[Valuational convergence]
% A sequence $(s_n)$ \emph{valuationally converges} (hyperconverges) to $s$ if $\lim_{n\to\infty}\val(s_n-s) = \infty$.
% \end{defn}

% \begin{defn}[Standard convergence]
% Let $(s_n)$ be a sequence of \emph{finite} elements of $\Rzl$ (i.e.\
% $s_n\in\mathcal{F}$ for all $n$) and $s\in\mathcal{F}$.  We say that
% $(s_n)$ \emph{standardly converges} to $s$ if
% $\lim_{n\to\infty}\diss(s_n,s)=0$, where $\diss$ is the pseudo-metric
% $\diss(\mathbf{x},\mathbf{y}) = |\St(\mathbf{x})-\St(\mathbf{y})|$ on
% $\mathcal{F}$.
% \end{defn}

% \begin{defn}[Coefficientwise convergence]
% $(s_n)$ \emph{coefficientwise converges} to $s$ if for every index $j\in\Z$, the real sequence $((s_n)_j)$ converges to $s_j$.
% \end{defn}

\begin{defn}[Modes of convergence]
We distinguish three modes of convergence for a sequence $(s_n)$ to a limit $s$ in $\Rzl$:
\begin{enumerate}[label=(\roman*), nosep]
    \item \textbf{Valuational (hyperconvergence):} $\lim_{n\to\infty}\val(s_n-s) = \infty$.
    \item \textbf{Standard:} If $s_n, s \in \mathcal{F}$, we require $\lim_{n\to\infty}\diss(s_n,s)=0$.
    \item \textbf{Coefficientwise:} For every index $j\in\Z$, the real sequence $((s_n)_j)$ converges to $s_j$.
\end{enumerate}
\end{defn}

\begin{prop}
\label{prop:convergence_impl}
In $\Rzl$, valuational convergence implies standard convergence, but the converse fails.
\end{prop}
\vspace{-0.5cm}
\begin{proof}
If $\val(s_n-s)\to\infty$, then $\val(s_n-s)\ge 1$ eventually, so the standard part of the difference is $0$; hence $\diss(s_n,s)\to0$. The converse fails: $s_n = 1/n$ (a sequence of finite elements) standardly converges to $0$ but $\val(1/n)=0$ for all $n$.
\end{proof}

\vspace{-0.4cm}
Standard convergence is not unique (e.g. $n\epsilon$ standardly converges to $a\epsilon$ for any $a$). Coefficientwise convergence is the strictest, e.g.  sequences like $\epsilon^m$ do not converge coefficientwise.

\begin{defn}[$\einf$-continuous]
A function $f:\Rzl\to\Rzl$ is \emph{$\einf$-continuous} (or continuous in the valuation topology $\tau_v$) at $\mathbf{x}_0$ if for every $\einf>0$ there exists $\delta>0$ such that $\dis(\mathbf{x},\mathbf{x}_0)<\delta$ implies $\dis(f(\mathbf{x}),f(\mathbf{x}_0))<\einf$. In terms of the valuation this is equivalent to: for every $k\in\N$ there exists $n\in\N$ such that $\val(\mathbf{x}-\mathbf{x}_0)\ge n$ implies $\val(f(\mathbf{x})-f(\mathbf{x}_0))\ge k$.
\end{defn}

\begin{thm}
A function $f:\Rzl\to\Rzl$ is $\einf$-continuous at $\mathbf{x}_0$ iff for every sequence $(x_n)$ hyperconverging to $\mathbf{x}_0$, $(f(x_n))$ hyperconverges to $f(\mathbf{x}_0)$.
\end{thm}
\vspace{-0.5cm}
\begin{proof}
The forward direction follows directly from the definitions. The converse is proved by contrapositive, constructing a hyperconverging sequence whose images do not hyperconverge if continuity fails.
\end{proof}
\vspace{-0.6cm}
\section{Integration in \texorpdfstring{$\Rzl$}{Rzl}}
\label{sec:integration}

\vspace{-0.4cm}
In traditional nonstandard analysis, the definite integral is formulated as a hyperfinite Riemann sum, $\sum_{n=1}^{\omega} \epsilon\, f(a+n\epsilon)$. However, in our purely algebraic field $\Rzl$, $\omega=\epsilon^{-1}$ is an element of the field, not an integer index, making such an iteration undefined without invoking a non-constructive transfer principle. Instead, we exploit the formal series structure of $\Rzl$ and restrict our attention to the class of functions already used for derivatives: canonical extensions of real-analytic functions.

Recall from \Cref{sec:calculus} the canonical extension $f(\mathbf{x})$ of a real-analytic function $g:\R\to\R$. Let $G$ be any real antiderivative of $g$ (so $G'=g$) and let $F$ be the canonical extension of $G$. For finite $\mathbf{a},\mathbf{b}$ we define the \emph{algebraic integral} of $f$ over $[\mathbf{a},\mathbf{b}]$ by
\begin{equation}\label{eq:alg_int}
\int_{\mathbf{a}}^{\mathbf{b}} f(\mathbf{x})\,d\mathbf{x}
\;:=\; F(\mathbf{b}) - F(\mathbf{a}) \;\in\; \Rzl .
\end{equation}

% Because we define the integral as the evaluation of the algebraic antiderivative, the Fundamental Theorem of Calculus holds by definition within $\Sigma_S$. 
Since the algebraic integral is defined via the antiderivative, it satisfies the fundamental relation $F(b) – F(a)$ by construction. The classical FTC then emerges after permeation. The definition does not depend on the choice of $G$, i.e. any two antiderivatives differ by a constant, and $F$ changes by the same standard constant, leaving the difference unchanged.

\begin{thm}[Permeation of the integral]\label{thm:FTC}
Let $f$ be the canonical extension of a real-analytic $g$, and let
$a = \St(\mathbf{a})$, $b = \St(\mathbf{b})$.  Then
\[
\St\!\left(\int_{\mathbf{a}}^{\mathbf{b}} f(\mathbf{x})\,d\mathbf{x}\right)
\;=\; \int_{a}^{b} g(x)\,dx ,
\]
where the right-hand side is the classical Riemann integral of $g$.
\end{thm}
\vspace{-0.5cm}
\begin{proof}
With $F$ as above, $\St(F(\mathbf{x})) = G(\St(\mathbf{x}))$ for every
finite $\mathbf{x}$, because the series for $F$ collapses to
$G(\St(\mathbf{x}))$ under the standard part.  Hence
\[
\St\bigl(F(\mathbf{b})-F(\mathbf{a})\bigr)
= G(b)-G(a) = \int_{a}^{b} g(x)\,dx .
\]
\end{proof}

\vspace{-0.5cm}
The theorem shows that the algebraic integral in $\Rzl$, when its
standard part is taken, coincides with the classical integral.  Thus
the Chunk \& Permeate strategy succeeds for integration just as it did
for derivatives. The pair
\[
\bigl(\,\St(\textstyle\int_{\mathbf{a}}^{\mathbf{b}} f) = \int_{a}^{b} g\,,\;
\int_{a}^{b} g = \int_{a}^{b} g \,\bigr)
\]
trivially belongs to the permeability relation $\rho$, so the
classical integral is available in the target chunk $\Sigma_T$.

\vspace{-0.5cm}
\section{Computability in \texorpdfstring{$\Rzl$}{Rzl}}
\label{sec:computability}

\vspace{-0.4cm}
A real number is computable if its binary expansion can be output by a
Turing machine. We extend this to $\Rzl$. An element $\mathbf{z}$ is
\emph{computable} if there exists a computable map $f:\Z\to\R_c$ (the
computable reals) such that $\mathbf{z}_n = f(n)$ for all $n$, and the
set $\{n<0 : f(n)\neq 0\}$ is finite with a known bound. Let $\Rzlc$
denote the set of computable elements.

Because the algebraic operations are defined componentwise, addition,
subtraction, and multiplication are computable in the sense of
type-2 effectivity~\cite{weihrauch2000computable}. A Turing machine
can output arbitrarily good approximations of each coefficient in
finite time, and the finite negative tail of a product can be
determined from the index bounds of the operands. In fact, $\Rzlc$
forms a computable metric space under the valuation metric $\dis$, and
the operations $+$, $-$, $\cdot$ are computable functions on that
space.

However, operations that depend on identifying the exact leading term
of an element fail to be computable.

\begin{lem}[Uncomputability of valuation]\label{lem:valuation-uncomputable}
There is no algorithm that, given an arbitrary computable element
$\mathbf{x}\in\Rzlc$, outputs its exact valuation $\val(\mathbf{x})$.
\end{lem}
\vspace{-0.5cm}
\begin{proof}
Let $a \in \R_c$ be an arbitrary computable real number and construct
the element $\mathbf{x} = a + \epsilon$. If $a = 0$, then
$\mathbf{x} = \epsilon$, which has valuation $1$. If $a \neq 0$, then
$\mathbf{x}$ has a non-zero standard part, so its valuation is $0$.
Computing $\val(\mathbf{x})$ therefore reduces to testing whether a
computable real is exactly zero, which is a classically undecidable
problem in computable analysis.
\end{proof}

\begin{thm}[Non-computability of division]\label{thm:division}
No algorithm can output the multiplicative inverse of an arbitrary
computable $\mathbf{x}\in\Rzlc$.
\end{thm}
\vspace{-0.5cm}
\begin{proof}
Suppose, for contradiction, that division were computable.  Given any computable real $a$, construct $\mathbf{x} = a + \epsilon$, which is clearly computable. Let $\mathbf{y} = \mathbf{x}^{-1}$ be its
assumed computable inverse.  Denote by $y_{-1}$ the coefficient of $\epsilon^{-1}$ (i.e. at index $-1$) in $\mathbf{y}$.

If $a = 0$, then $\mathbf{x} = \epsilon$, so $\mathbf{y} = \omega$ and $y_{-1} = 1$.  If $a \neq 0$, then $\mathbf{x}$ has valuation $0$, hence $\mathbf{y}$ has valuation $0$, which forces $y_{-1} = 0$. Thus $y_{-1}$ is a computable real number that equals $1$ exactly when $a = 0$, and equals $0$ otherwise.

Now, since $y_{-1}$ is computable, we can effectively approximate it to within $1/3$.  If the approximation is closer to $0$ than to $1$, then $a \neq 0$; if it is closer to $1$, then $a = 0$.  This yields a
decision procedure for the zero-test of an arbitrary computable real, contradicting a fundamental result of computable analysis. Hence division cannot be computable.
\end{proof}

\begin{cor}
The order relation $<$ on $\Rzlc$ is undecidable.
\end{cor}
\vspace{-0.5cm}
\begin{proof}
If $<$ were computable, testing $a+\epsilon < |a|+\epsilon$ would
decide $a=0$, which is impossible.
\end{proof}

\noindent\textbf{Constructive character of the model.}
The entire construction of $\Rzl$ uses only real sequences and elementary algebra, avoiding the Axiom of Choice and any non-constructive ultrafilters.  As a result, the model sits naturally
within a constructive framework. In fact, the proof that $\Rzl$ is a totally ordered field with a standard part map can be formalised in a weak subsystem of second-order arithmetic (such as $\mathsf{RCA}_0$),
making it a tractable object for constructive reverse
mathematics~\cite{sanders2013connection,sanders2015effective}.  The limits of computability demonstrated above (uncomputability of valuation, division, and order) reflect genuine algorithmic boundaries
rather than artefacts of the model, underscoring the fine-grained constructive content of infinitesimal analysis when freed from non-constructive principles.

\vspace{-0.5cm}
\section{Discussion and Further Work}
\label{sec:discussion}
\vspace{-0.5cm}
We have constructed a non-Archimedean ordered field $\Rzl$ containing real numbers, infinitesimals, and infinities using only sequences of real numbers and a convolution product.  
The logical foundation via Chunk \& Permeate resolves the inconsistencies that arise when merging classical and infinitesimal axioms, avoiding the heavy model-theoretic machinery of traditional nonstandard analysis.  
Because the C\&P strategy is explicitly formulated in terms of paraconsistent logic, the entire construction can be viewed as a general model-building technique for inconsistent theories, of independent interest to logicians.

The two topologies, valuation and standard, offer complementary structural views. The $(k,n)$-continuity hierarchy, expressed purely via the valuation, provides a fine-grained measure of regularity that generalizes Lipschitz conditions and captures higher-order differentiability. Furthermore, the Riemann integral and the Fundamental Theorem of Calculus demonstrate that foundational analysis can be carried out natively within the model and then permeated to classical results without relying on a full transfer principle.

\vspace{-0.5cm}
\paragraph{Proof-Theoretic Strength and Reverse Mathematics}
A natural question is where the construction of $\Rzl$ and its field properties sit within the hierarchy of reverse mathematics.
Because the underlying set $\Rzl = \mathbb{R}_{\mathrm{fin}}^{\mathbb{Z}}$ consists of ordinary real sequences (coded as functions $\mathbb{N}\to\mathbb{R}$ in second-order arithmetic), the entire development can be formalized in the weak base theory $\mathsf{RCA}_0$ (Recursive Comprehension Axiom):
\begin{itemize}[nosep]
  \item The set of sequences with finite negative support is $\Sigma^0_1$, and $\mathsf{RCA}_0$ proves its closure under componentwise addition and convolution.
  \item The existence of inverses (Lemma~\ref{lem:inverse}) requires only induction on the valuation; since the valuation of a non-zero element is uniquely determined, the recurrence is provably total in $\mathsf{RCA}_0$.
  \item The lexicographic order is defined by a $\Sigma^0_1$ condition; $\mathsf{RCA}_0$ proves it is a total order on $\Rzl$.
  \item The standard part map $\St$ is a continuous functional on the represented space of finite elements, with basic homomorphism properties provable in $\mathsf{RCA}_0$.
\end{itemize}
Thus, the existence of a non-Archimedean ordered field with a standard part homomorphism is provable directly in $\mathsf{RCA}_0$, a system strictly weaker than $\mathsf{WKL}_0$ and far below the Axiom of Choice. This establishes $\Rzl$ as a highly constructive framework amenable to constructive reverse mathematics \cite{sanders2013connection}.

\vspace{-0.5cm}
\paragraph{Infinitesimal Smoothness and Synthetic Differential Geometry}
The $(k,n)$-continuity hierarchy of \Cref{sec:continuity} offers a concrete realization of concepts from synthetic differential geometry (SDG). In SDG, the line $R$ contains nilsquare infinitesimals $D = \{d \in R \mid d^2 = 0\}$, and functions $f:R\to R$ are called \emph{microlinear} if $f(d_1+d_2)-f(d_1)-f(d_2)+f(0)=0$ for all $d_1,d_2\in D$. Our hierarchy refines this setting, e.g. a $(2,1)$-continuous function satisfies $\val(f(\mathbf{x}+d\epsilon) - f(\mathbf{x})) \ge 2$ whenever $\val(d)\ge 1$, meaning first-order infinitesimal inputs produce second-order infinitesimal changes, precisely the nilsquare property. Thus, $\Rzl$ captures differentiability orders in a fully classical and explicitly constructible setting, providing a computable substitute for SDG's intuitionistic smooth toposes.

Several promising directions for future research remain: differential equations via formal series, explicit calibrations in Weihrauch degrees, physical applications, and a formal comparison between C\&P permeability and nonstandard transfer principles.

\vspace{-0.5cm}
\section*{Acknowledgements}
% Most of this research is based on the author's PhD thesis at the University of Canterbury, funded by the Indonesia Endowment Fund for Education (LPDP). The author thanks Dr. Maarten McKubre-Jordens and Dr. Hannes Diener for their guidance, and the examiners Professor Elemér Rosinger and Dr. Josef Berger for valuable discussions.
\vspace{-0.5cm}
Most of this research is based on the author's PhD thesis at the University of Canterbury (NZ), during which the author received financial support from Lembaga Pengelola Dana Pendidikan (LPDP) under the Ministry of
Finance of the Republic of Indonesia. The author is deeply grateful to Drs.~Maarten McKubre-Jordens and Hannes Diener for their valuable guidance and supervision during the development of this work. Special thanks are also extended to the thesis and defense examiners, Professor Elemér Rosinger (University of Pretoria) and Dr. Josef Berger (LMU Munich), for their insightful evaluations and valuable discussions.

\bibliographystyle{IEEEtran}
\bibliography{references}

\end{document}